\documentclass[12pt]{article}
\usepackage{amsmath,amsthm,amssymb,latexsym,color, dsfont}
\usepackage{epsfig}
\usepackage{latexsym}

\def\la{\left\langle}
\def\ra{\right\rangle}

\newcommand{\km}{k\mbox{-}\min}
\newcommand{\kma}{k\mbox{-}\max}
\newcommand{\nkma}{(n-k+1)\mbox{-}\max}

\newcommand{\PP}{\mathbb{P}}
\newcommand{\R}{\mathbb{R}}
\newcommand{\E}{\mathbb{E}}

\newtheorem{theo}{Theorem}[section]
\newtheorem{lem}[theo]{Lemma}

\theoremstyle{definition}
\newtheorem{rem}[theo]{Remark}

\expandafter\let\expandafter\oldproof\csname\string\proof\endcsname
\let\oldendproof\endproof
\renewenvironment{proof}[1][\proofname]{%
  \oldproof[\bf #1]%
}{\oldendproof}

\date{}

\title{Estimates for order statistics in terms of quantiles}
\author{Alexander E. Litvak and Konstantin Tikhomirov}

\begin{document}

\maketitle

\begin{abstract}
Let  $X_1, \ldots, X_n$ be  independent non-negative random variables with
cumulative distribution functions $F_1,F_2,\ldots,F_n$, each satisfying
certain (rather mild) conditions. We show that the median of $k$-th smallest
order statistic of the vector $(X_1, \ldots, X_n)$ is equivalent to
the quantile of order $(k-1/2)/n$ with respect to the averaged distribution
$F=\frac{1}{n}\sum_{i=1}^n F_i$.
\end{abstract}



\section{Introduction}

The goal of this note is to provide sharp estimates for order statistics of independent, not necessarily
identically distributed random variables, whose distributions satisfy certain (rather mild) conditions.
Order statistics are among very important objects in probability and statistics with many applications.
We refer to \cite{AN, BCo, DN} and references therein for information on the subject, especially in the case
of i.i.d.\ random variables.
The case of independent but not identically distributed random variables is less studied, we refer to \cite[Chapter~5]{DN}
for some results in this direction.
Understanding this setting is important in some applications, for example in connection with
the Mallat--Zeitouni problem \cite{MZ, LTik},
the study of asymptotic behaviour of some classes normed spaces \cite{GLSW-Bull},
some problems in reconstruction \cite{GLMP}, to name a few.

Given $1\leq k\leq n$ and a sequence of real numbers $a_1, a_2, \ldots, a_n$, let $\km _{i\leq n}a_{i}$ and
$\kma _{i\leq n}a_{i}$ denote its $k$-th smallest and $k$-th  largest elements, in particular,
$$
  \km _{i\leq n}a_{i} = \nkma _{i\leq n}a_{i}.
$$

Let $F$ be cdf (cumulative distribution function) of a non-negative random variable. We employ the following condition:
\begin{equation}\label{cdf specific conditions}
\mbox{{\it there exists  $\, K>1\, $ such that $\, \frac{F(Kt)}{1-F(Kt)}\geq \frac{2F(t)}{1-F(t)}\, \, $ for all $\, t>0$}}
\end{equation}
(see the next section for discussion and examples).

\medskip

The main result of this note, Theorem~\ref{median estimation indep}, states
that given $K>1$, $1\leq k\leq n$, and independent non-negative random variables $X_1, \ldots, X_n$
with cdf's $F_1,F_2,\ldots,F_n$, each satisfying condition
\eqref{cdf specific conditions} with parameter $K$, one has
$$
  K^{-10}\,  q_F\left(\frac{k-1/2}{n}\right) \leq  \mbox{\rm Med}\Big( \km _{1\leq i\leq n}X_i \Big)
  \leq K^{13}\,  q_F\left(\frac{k-1/2}{n}\right),
$$
where $q_F(t)$ is the quantile of order $t$ with respect to the averaged distribution $F=\frac{1}{n}\sum_{i=1}^n F_i$.

This result improves and complements the results from \cite{GLSW-CRAS, GLSW-PAMS, GLSW-P}, where, under
somewhat stronger conditions on distributions, the authors proved estimates for the corresponding
expectations up to a factor logarithmic in $k$. More precisely, in \cite{GLSW-CRAS, GLSW-PAMS} it was shown that
given
$\alpha, \beta, p >0$, $1\leq k\leq n$, real numbers $0<x_1\leq x_2\leq \ldots \leq x_n$, and independent
random variables $\xi _1, \dots, \xi_n$ satisfying
$$
 \forall t>0 \quad  \PP  \left(|\xi | \leq t \right) \leq  \alpha t \quad \mbox{ and } \quad
  \PP  \left(|\xi | > t \right) \leq e^{-\beta t}
$$
one has
$$
   \frac{1}{2^{1/p}\, 4\, \alpha}\,  \max_{1 \leq j \leq k}\ \frac{k+1-j}{
   \sum_{i=j}^n 1/x_i} \leq
   \left(\E\, \km_{1\leq i\leq n} |x_{i}\xi _{i}|^p \right)^{1/p} \leq
  C (p, k)   \, \beta ^{-1} \, \max_{1 \leq j \leq k} \
  \frac {k+1-j}{\sum_{i=j}^n 1/x_i } ,
$$
where $C(p, k) := C \max\{ p,  \ln (k+1) \}$,
and $C$ is an absolute positive constant. In \cite{GLSW-P} this was extended further
to a larger class of distributions, namely
it was shown that the expectation above is equivalent to some Orlicz
norm of the sequence $(1/x_i)_i$, again up to a factor logarithmic in $k$.

We would also like to mention that
order statistics of random vectors with independent but not identically distributed coordinates
were studied in \cite{S70}, where a result of Hoeffding \cite{H56}
was used, in particular, to estimate the difference between the median of
$\km _{1\leq i\leq n}(X_i)$ and
the median of the $k$-th order statistic of a random vector with i.i.d.\ coordinates
distributed according to the law $F=\frac{1}{n}\sum_{i=1}^n F_i$
(see also \cite[pp 96--97]{DN}).
However, the results of \cite{S70} do not seem to directly
imply the relations which we prove in Theorem~\ref{median estimation indep}.

\section{Notation and preliminaries}
\label{secone}

Given a subset $A\subset \mathbb N$, we denote its cardinality by $|A|$.
Next, for a natural number $n$ and a set $E\subset \{1,2,\dots,n\}$,
we denote by $E^c$ the complement of $E$ inside $\{1,2,\dots,n\}$.
Similarly, for an event ${\mathcal E}$ we denote by ${\mathcal E}^c$ the complement of the event.
Further, we say that a collection of sets $(A_j)_{j\leq k}$ is a partition of $\{1,2,\dots,n\}$ if
each $A_j$ is non-empty, the sets are pairwise disjoint and their union is $\{1,2,\dots,n\}$.
The canonical Euclidean norm and the canonical inner product in
$\R^n$ will be denoted by $|\cdot|$ and $\la \cdot, \cdot \ra $, respectively.
We adopt the conventions $1/0=\infty$ and $1/\infty=0$ throughout the text.

Let $\xi$ be a real-valued random variable.
As usual, we use the abbreviation cdf for the cumulative distribution function
(that is, the cdf of $\xi$ is $F_{\xi}(t)=\PP(\xi \leq t)$).
Given $r\in [0,1]$, by $q(r)=q_F(r)=q_{\xi} (r)$ we denote a quantile
of order $r$, that is a number satisfying
$$
   \PP\left\{\xi<q(r)\right\} \leq r \quad \mbox{ and } \quad
   \PP\left\{\xi \leq q(r)\right\} \geq r
$$
(note that in general $q(r)$ is not uniquely defined).

Now we discuss our main condition on the distributions, the condition (\ref{cdf specific conditions}).
Clearly, if the cdf of a non-negative random variable $\xi$ satisfies condition \eqref{cdf specific conditions}
with some $K$ then for every $x>0$ the cdf of $x \xi$ satisfies \eqref{cdf specific conditions} with the same
$K$.
Note that  \eqref{cdf specific conditions} is equivalent to
\begin{equation}\label{cdf-cond-3}
\mu\bigl((t, Kt]\bigr) \geq \mu\bigl([0, t]\bigr)\, \mu\bigl((Kt, \infty)\bigr),\quad t>0,
\end{equation}
where $\mu$ is the probability measure on $\R$ (actually, on $\R_+$) induced by $F$.
It is not difficult to see that the uniform distribution on $[0,1]$ satisfies the condition
\eqref{cdf specific conditions} with $K=2$. Another example of a random variable satisfying
\eqref{cdf specific conditions} (with $K=2^{1/p}$) is a random variable $\xi$ taking values
in $[1, \infty)$ with $\PP(\xi \geq 1)=1/t^p$, $t\geq 1$, where $p>0$ is a fixed parameter.
Next we show that the absolute value of any log-concave random variable
satisfies~\eqref{cdf specific conditions}.
In particular, this includes Gaussian and exponential distributions.

\begin{lem}\label{l-c-cond}
Let $\eta$ be a log-concave variable. Then the cdf of $|\eta|$ satisfies~\eqref{cdf specific conditions} with $K=3$.
\end{lem}

The lemma is an immediate consequence of the following statement
and the fact that conditions \eqref{cdf specific conditions} and \eqref{cdf-cond-3} are equivalent.

\begin{lem}
Let $\mu_0$ be a non-degenerate log-concave probability measure on $\R$ and let $t>0$. Then
$$\mu_0\bigl((t, \infty)\bigr)\, \mu_0\bigl((t, 3t]\bigr) \geq \mu_0\bigl([-t, t]\bigr)\,  \mu_0\bigl((3t, \infty)\bigr)$$
and
$$\mu_0\big((-\infty, -t)\big)\,  \mu_0\big([-3t, -t)\big) \geq    \mu_0\big([-t, t]\big) \,  \mu_0\bigl((-\infty, -3t)\bigr).$$
In particular, we have
\begin{align*}
\mu\big((t,3t]\big)=\mu_0\big([-3t,-t)\cup (t,3t]\big)&\geq \mu_0\bigl([-t, t]\bigr)\,\mu_0\bigl((-\infty, -3t)\cup (3t,\infty)\bigr)\\
&=\mu\bigl([0, t]\bigr)\,\mu\bigl((3t,\infty)\bigr),
\end{align*}
where $\mu$ is defined by $\mu(S):=\mu_0(-S\cup S)$, $S\subset\R_+$.
\end{lem}

\begin{proof}
We prove the first inequality only, the second one is similar.
Note that
$$
    (t, \infty) =  \tfrac{1}{2} \, [-t, \infty) + \tfrac{1}{2} \, (3t, \infty) .
$$
By log-concavity of $\mu_0$ this implies
$$
  {\mu_0} ^2 \big((t, \infty)\big) \geq \mu_0\big([-t, \infty)\big)\,  \mu_0\big((3t, \infty)\big) =
  \big(\mu_0\big([-t, t]\big) + \mu_0\big((t, \infty)\big)\big)\,  \mu_0\big((3t, \infty)\big).
$$
Thus
$$
  \mu_0\big((t, \infty)\big)\, \left( \mu_0  \big((t, \infty)\big) - \mu_0\big((3t, \infty)\big)\right) \geq
 \mu_0\big([-t, t]\big) \, \mu_0\big((3t, \infty)\big),
$$
which implies the result.
\end{proof}

\begin{rem}
We would also like to notice that \eqref{cdf specific conditions} implies that
$$
   F(t) \geq 2 F(t/K^2),\quad \quad  \mbox{ whenever }\quad \quad  F(t)\leq 1/2.
$$
This (weaker) assumption on $F$ was employed in \cite{LTik}.
\end{rem}

\section{Main result}
\label{mainth}

In this section we prove our main result, stating that medians of order statistics in case
of independent components are equivalent to corresponding quantiles of an averaged distribution.

\begin{theo}\label{median estimation indep}
Let $K>1$ and $k\leq n$.
Let  $X_1, \ldots, X_n$ be  independent non-negative random variables
with cdf's $F_1,F_2,\ldots,F_n$, each satisfying condition
\eqref{cdf specific conditions} with parameter $K$. Set $F=\frac{1}{n}\sum _{i=1}^n F_i$.
Then for  $0<t<K^{-5}$ one has
$$
  \PP\Bigl\{\km _{1\leq i\leq n} X_i< t\, q_F\Big(\frac{k-1/2}{n}\Big) \Bigr\} \leq
   4\, t^{1/(4 \ln K)},
$$
and for $t>K^5$ one has
$$
  \PP\Bigl\{\km _{1\leq i\leq n} X_i> t\, q_F\Big(\frac{k-1/2}{n}\Big) \Bigr\} \leq
   4\, t^{-1/(6 \ln K)}.
$$
In particular,
$$
  K^{-10}\,  q_F\left(\frac{k-1/2}{n}\right) \leq  \mbox{\rm Med}\Big( \km _{1\leq i\leq n}X_i \Big)
  \leq K^{13}\,  q_F\left(\frac{k-1/2}{n}\right).
$$
\end{theo}

In the proof of the theorem, we will use two following auxiliary statements.
\begin{lem}
Let $F:(0, \infty) \to [0, 1]$ be a non-decreasing function satisfying \eqref{cdf specific conditions}.
Let $\ell \geq 1$, $\gamma\in (0,1)$, and $t>0$.
Then
\begin{equation}\label{gamma-one}
  F(t) \geq 2^{\ell} (1-F(t)) \, F\left(t/K^{\ell}\right)
\end{equation}
and, assuming that $F(t)\geq 1-\gamma$,
\begin{equation}\label{gamma-two}
  1-F\left(t/K^{\ell}\right) \geq \frac{2^\ell}{2^\ell \gamma +1} \, (1-F(t)) .
\end{equation}
\end{lem}
\begin{proof}
Applying \eqref{cdf specific conditions} $\ell$ times we obtain
$$
   \frac{F(t)}{1-F(t)} \geq \frac{2^\ell \, F(t/K^\ell)}{1-F(t/K^\ell)},
$$
which implies \eqref{gamma-one}.
Fix a parameter $\beta\in (0,1)$, which will be specified later. If
$F\left(t/K^{\ell}\right)\geq \beta$ then the above inequality implies
$$
     1-F\left(t/K^{\ell}\right) \geq 2^\ell\, \beta \, (1-F(t)) .
$$
Otherwise, if $F\left(t/K^{\ell}\right)< \beta$, we get
$$
  1-F\left(t/K^{\ell}\right) \geq 1-\beta \geq \frac{1-\beta}{\gamma}(1-F(t)).
$$
Choosing $\beta:= 1/(2^\ell \gamma +1)$, we get \eqref{gamma-two} and complete the proof.
\end{proof}

\bigskip

The next  simple lemma can be verified by considering
the expectation and the variance
of the sum of random Bernoulli variables and using the Chebyshev inequality.

\begin{lem}\label{bernoulli concentration lemma}
Let $\eta_1,\ldots,\eta_n$ be independent Bernoulli $0/1$ random variables with probabilities of success
$p_1,p_2,\dots,p_n$. Then for every $t>0$ we have
$$
   \PP\Bigl\{\Bigl|\sum\limits_{i=1}^n \eta_i-\sum\limits_{i=1}^n p_i\Bigr|
   \geq t\Bigr\}\leq\frac{1}{t^2}\, \sum\limits_{i=1}^n p_i.
$$
\end{lem}

\medskip

\begin{proof}[Proof of Theorem~\ref{median estimation indep}.]
We start with the first bound.
Take any positive $q<q_F\big(\frac{k-1/2}{n}\big)$. By definition of the quantile, we have
$\sum_{i=1}^n F_i(q)\leq k-1/2.$
To estimate $\km_{i \leq n} X_i$ from below it is enough to show
that the set of indices $i$ corresponding to ``small'' $X_i$'s has cardinality at most $k-1$.

\smallskip

Fix $\ell \geq 5$ such that $K^{-\ell -1}\leq t <K^{-\ell}$ and put $\gamma := 1/2^{\ell/2}$, $t_0:=q/K^\ell$.
Further, set
$$
   I:=  \{i\leq n \, : \, F_i(q) < 1-\gamma\} \quad \mbox{ and } \quad
   I^c:=  \{i\leq n \, : \, F_i(q) \geq 1-\gamma\}.
$$
We want to estimate the number of indices $i\in I$ corresponding to ``small'' $X_i$.
Denote
$$
   A:=\sum_{i\in I}F_i(q) \quad \mbox{ and } \quad B:=\sum_{i\in I}F_i(t_0).
$$
Applying \eqref{gamma-one} to $F_i$, $i\in I$, we get that
$A\geq 2^\ell \gamma  B$. Therefore,
if $A\leq 2^\ell \, \gamma/(2^\ell \, \gamma - 1)$ then
$$
  \PP\left\{\exists i\in I \, :\, X_i\leq t_0\right\} \leq
  \sum _{i\in I} \PP\left\{X_i\leq t_0\right\}  = B
  \leq \frac{A}{2^\ell \, \gamma } \leq \frac{1}{2^\ell \, \gamma -1 }.
$$
If $A> 2^\ell \, \gamma/(2^\ell \, \gamma - 1)$ then,
applying  Lemma~\ref{bernoulli concentration lemma}, we get
\begin{align*}
\PP\big\{|\{i\in I:\,X_i\leq t_0\}|< A\big\}
&=\PP\Big\{ \sum_{i\in I} \chi _{\{X_i\leq t_0\}} < A\Big\}\\
&\geq 1- \frac{B}{(A-B)^2}\\
&\geq 1- \frac{2^\ell \, \gamma}{A(2^\ell \, \gamma - 1)^2}\\
&\geq 1- \frac{1}{2^\ell \, \gamma -1 }.
\end{align*}
Thus, in both cases we have
\begin{equation}\label{median indep 1}
\PP\Bigl\{|\{i\in I:\,X_i\leq t_0\}|< \sum_{i\in I}F_i(q)\Bigr\}\geq 1- \frac{1}{2^\ell \, \gamma -1 }.
\end{equation}

Next, we estimate the number of indices $i\in I^c$ corresponding to ``small'' $X_i$'s.
If $a:=\sum_{i\in I^c}(1-F_i(q))< 1/2$,
then we have
$$
  |I^c|< \frac{1}{2} + \sum_{i\in I^c}F_i(q) \leq  k-\sum_{i\in I}F_i(q),
$$
and from \eqref{median indep 1} we obtain
$$
 \PP\Bigl\{|\{i\leq n :\,X_i\leq t_0\}|< k \Bigr\}\geq 1- \frac{1}{2^\ell \, \gamma -1 }.
$$
Now, assume that $a \geq 1/2$. Set $b:=\sum_{i\in I^c}(1-F_i(t_0))$.
Applying \eqref{gamma-two} to $F_i$, $i\in I^c$, we get
$b\geq \frac{2^\ell}{2^\ell \gamma +1 } a$. Note that $\sum _{i\in I^c} F_i(q) = |I^c|-a$.
Therefore, by Lemma~\ref{bernoulli concentration lemma} we obtain
\begin{align*}
\PP\Big\{|\{i\in I^c:\, X_i \leq t_0\}|< \sum _{i\in I^c} F_i(q) \Big\}
&= \PP\Big\{\sum_{i\in I^c} \chi _{\{X_i\leq t_0\}} < |I^c|-a \Big\}\\
&= \PP\Big\{b-\sum_{i\in I^c} \chi _{\{X_i> t_0\}} < b-a \Big\}\\
&\geq 1 - \frac{b}{(b-a)^2}\\
&\geq 1-  \frac{2\cdot 2^\ell (2^\ell \gamma +1 )}{(2^\ell -2^\ell \gamma - 1)^2}.
\end{align*}
Combining the last relation with \eqref{median indep 1} and using that $\sum_{i\leq n} F_i(q) \leq k-1/2<k$,  we obtain
$$
  \PP\Bigl\{|\{i\leq n :\,X_i\leq t_0\}|< k \Bigr\}\geq 1- \frac{1}{2^\ell \, \gamma -1 } -
   \frac{2\cdot 2^\ell (2^\ell \gamma +1 )}{(2^\ell -2^\ell \gamma -1)^2}  \geq 1- \frac{4}{2^{\ell/2}},
$$
where in the last inequality we used the assumption $\ell \geq 5$ and the identity $\gamma = 1/2^{\ell/2}$. This proves
$$
  \PP\Bigl\{\km _{1\leq i\leq n} X_i\leq q/K^\ell \Bigr\} \leq  \frac{4}{2^{\ell/2}}.
$$
Finally, by the choice of $\ell$ we have $\ell \geq (4\ln (1/t))/(5\ln K)$,
which implies the first part of the theorem.

\bigskip

The second part is somewhat similar.
To make comparison with the first part of the proof straightforward, we will use the same letters
for corresponding sets or numbers, just adding a bar.
Let $\bar q:=q_F\left(\frac{k-1/2}{n}\right)$.
By definition, we have
$\sum_{i=1}^n F_i(\bar q)\geq k-1/2.$
To estimate $\km_{i \leq n} X_i$ from above we will show that
the set of indices $i$ corresponding to ``small'' $X_i$ typically has cardinality at least $k$.
Fix $\bar \ell \geq 5$ such that
such that $K^{\bar \ell}\leq t <K^{\bar \ell+1}$, and set $\bar \gamma := 2/4^{\bar \ell/3}$, $\bar t_0:=K^{\bar \ell} \bar q$.
Further, let
$$
   \bar I:=\{i\leq n \, : \, F_i(\bar t_0) < 1-\bar \gamma\} \quad \mbox{ and } \quad
   \bar I^c:=\{i\leq n \, : \, F_i(\bar t_0) \geq 1-\bar \gamma\}.
$$
Let us bound the number of indices $i\in \bar I$ corresponding to ``small'' $X_i$.
Denote
$$
   \bar A:=\sum_{i\in \bar I}F_i(\bar q) \quad \mbox{ and } \quad \bar B:=\sum_{i\in \bar I}F_i(\bar t_0).
$$
Assume that $\bar A\geq 1/2$.
Applying  Lemma~\ref{bernoulli concentration lemma}, we get
\begin{align*}
\PP\big\{|\{i\in \bar I:\,X_i\leq \bar t_0\}| >  \bar A\big\} &= \PP\Big\{ \sum_{i\in \bar I} \chi _{\{X_i\leq \bar t_0\}} > \bar A\Big\}\\
&=\PP\Big\{\bar B - \sum_{i\in \bar I} \chi _{\{X_i\leq \bar t_0\}} < \bar B- \bar A\Big\}\\
&\geq 1- \frac{\bar B}{(\bar B-\bar A)^2}\\
&\geq 1- \frac{2^{\bar \ell +1} \, \bar \gamma}{(2^{\bar \ell} \, \bar \gamma - 1)^2},
\end{align*}
where we used the estimate $\bar B\geq 2^{\bar \ell}\bar \gamma\bar A$,
which follows from \eqref{gamma-one}.
Thus, in both cases $\bar A< 1/2$ and $\bar A\geq 1/2$ we have
\begin{equation}\label{median indep 5}
    \PP\Bigl\{|\{i\in \bar I:\,X_i\leq\bar t_0\}|> \sum_{i\in \bar I}F_i(\bar q)-1/2 \Bigr\}
     \geq 1- \frac{2^{\bar \ell +1} \, \bar \gamma}{(2^{\bar \ell} \, \bar \gamma - 1)^2}.
\end{equation}

Next, we estimate the number of indices $i\in \bar I^c$ corresponding to ``small'' $X_i$'s. Fix
$$
 \lambda:= \sqrt{2^{\bar \ell} (2^{\bar \ell} \bar \gamma +1)}/(2^{\bar \ell} - 2^{\bar \ell}\bar \gamma -1) < 1.
$$
If $\bar a:=\sum_{i\in \bar I^c}(1-F_i(\bar q))< \lambda$, then
$
  |\bar I^c|< \lambda + \sum_{i\in \bar I^c}F_i(\bar q) .
$
In this situation we have
$$
  |\{i\in \bar I^c:\, X_i \leq \bar t_0\}|\geq  \sum _{i\in \bar I^c} F_i(\bar q)
$$
if and only if $X_i \leq \bar t_0$ for all $i\in \bar I^c$. Note also that for every
$a_1, \ldots, a_m\in (0,1]$ one has
$$
  \mbox{ if } \quad \sum _{i=1}^m a_i \geq m-\lambda \quad
 \mbox{ then } \quad
  \prod _{i=1}^m a_i \geq 1- \lambda.
$$
This and independence of $X_i$'s imply
$$
  \PP\Big\{|\{i\in \bar I^c:\, X_i \leq \bar t_0\}|\geq  \sum _{i\in \bar I^c} F_i(\bar q) \Big\} =
   \prod _{i\in \bar I^c} \PP\big\{X_i\leq \bar t_0\big\}
  \geq \prod _{i\in \bar I^c} \PP\big\{X_i\leq \bar q\big\}
  \geq 1- \lambda.
$$
Together with \eqref{median indep 5}, it gives
$$
  \PP\Bigl\{|\{i\leq n :\,X_i\leq \bar t_0\}|> \sum _{i=1}^n F_i(\bar q) -1/2 \Bigr\}
  \geq 1-  \frac{2^{\bar \ell +1} \, \bar \gamma}{(2^{\bar \ell} \, \bar \gamma - 1)^2} - \lambda.
$$
It remains to consider the case $\bar a \geq \lambda$. Set
$$
  \bar b:=\sum_{i\in \bar I^c}(1-F_i(\bar t_0)).
$$
Applying \eqref{gamma-two} to $F_i$, $i\in \bar I^c$, we get that
$\bar a\geq \frac{2^{\bar \ell}}{2^{\bar \ell} \bar \gamma +1 } \bar b$. Note that $\sum _{i\in \bar I^c} F_i(\bar q) = |\bar I^c|-\bar a$.
Therefore, by Lemma~\ref{bernoulli concentration lemma}, we obtain
\begin{align*}
\PP\Big\{|\{i\in \bar I^c:\, X_i \leq \bar t_0\}|> \sum _{i\in \bar I^c} F_i(\bar q) \Big\}
&=\PP\Big\{\sum_{i\in \bar I^c} \chi _{\{X_i\leq \bar t_0\}} > |\bar I^c|-\bar a \Big\}\\
&=\PP\Big\{\sum_{i\in \bar I^c} \chi _{\{X_i> \bar t_0\}}-\bar b < \bar a - \bar b \Big\}\\
&\geq 1 - \frac{\bar b}{(\bar a-\bar b)^2}\\
&\geq 1-\frac{2^{\bar \ell} (2^{\bar \ell} \bar \gamma +1 )}{(2^{\bar \ell} -2^{\bar \ell} \bar \gamma - 1)^2\, \lambda}.
\end{align*}
Combining this with \eqref{median indep 5},
we obtain
$$
  \PP\Bigl\{|\{i\leq n :\,X_i\leq \bar t_0\}|> \sum _{i=1}^n F_i(\bar q) -1/2 \Bigr\}\geq
    1- \frac{2^{\bar \ell +1} \, \bar \gamma}{(2^{\bar \ell} \, \bar \gamma - 1)^2}-\frac{2^{\bar \ell}
 (2^{\bar \ell} \bar \gamma +1 )}{(2^{\bar \ell} -2^{\bar \ell} \bar \gamma - 1)^2\, \lambda}.
$$
Since $\bar \ell \geq 5$ and in view of the definitions of $\bar \gamma$ and $\lambda$, in both cases
$\bar a< \lambda$ and $\bar a\geq \lambda$
one has
$$
   \PP\Bigl\{|\{i\leq n :\,X_i\leq \bar t_0\}|> \sum _{i=1}^n F_i(\bar q) -1/2 \Bigr\}\geq
    1- \frac{4}{2^{\bar \ell/3}}.
$$
Note that $\sum _{i=1}^n F_i(\bar q) -1/2 \geq k-1$, thus the last estimate implies
$$
  \PP\Bigl\{\km _{1\leq i\leq n} X_i >  K^{\bar \ell} \, \bar q \Bigr\} \leq  \frac{4}{2^{\bar \ell/3}}.
$$
Finally, observe that by the choice of $\bar \ell $ we have $\bar\ell \geq (4\ln t)/(5\ln K)$,
which implies the second estimate in the theorem.
\end{proof}

\smallskip

\noindent
A. E. Litvak and K. Tikhomirov\\  {\small Dept.\ of Math.\ and Stat.\ Sciences},\\
{\small University of Alberta}, {\small Edmonton, AB, Canada T6G 2G1},\\
{\small \tt  aelitvak@gmail.com}\\
{\small \tt  ktikhomi@ualberta.ca}

\smallskip
\noindent {\small Current address of K.T.: Dept.\ of Math., Fine Hall, Princeton, NJ 08544}

\end{document}